\documentclass[10pt,a4paper]{amsart}
\usepackage[latin1]{inputenc}
\usepackage{latexsym}
\usepackage{color,graphicx,shortvrb}
\usepackage{amsmath, amssymb}
\usepackage{amsfonts}
\usepackage[colorlinks, bookmarks=true]{hyperref}

\newtheorem{theorem}{Theorem}[section]
\newtheorem{lemma}[theorem]{Lemma}
\newtheorem{proposition}[theorem]{Proposition}
\newtheorem{corollary}[theorem]{Corollary}

\theoremstyle{definition}
\newtheorem{definition}[theorem]{Definition}
\newtheorem{remark}[theorem]{Remark}

\author{J. M. Almira$^*$, Kh. F. Abu-Helaiel}
\title{On  Montel's theorem in several variables}
\thanks{$^*$ Corresponding author}

\begin{document}

%\keywords{}

%Ultrametric Banach spaces, Approximation theory, Lethargy theorems, p-adic transcendental numbers.}

%\subjclass[2010]{}

% 41A65, A1A25, 11J61, 11J81, 11K60.}

\begin{abstract} 
Recently, the first author of this paper, used the structure of finite dimensional translation invariant subspaces of $C(\mathbb{R},\mathbb{C})$ to give a new proof of classical Montel's theorem, about continuous solutions of Fr\'{e}chet's functional equation $\Delta_h^mf=0$, for real functions (and complex functions) of one real variable. In this paper we use similar ideas to prove a Montel's type theorem for the case of complex valued functions defined over the discrete group $\mathbb{Z}^d$. Furthermore, we  also state and demonstrate an improved version of  Montel's Theorem  for complex functions of several real variables and  complex functions of several complex variables.
\end{abstract}

\maketitle

\begin{quotation}
\noindent{\bf Key Words}: {Functional Equations,  Montel's Theorem, Invariant subspaces.}

\noindent{\bf 2010 Mathematics Subject Classification}:  47A15, 46F05, 46F10, 39B22, 39B32\\
   % Se\-con\-dary       xxyxx
\end{quotation}

\markboth{J. M. Almira, Kh. F. Abu-Helaiel}{On  Montel's theorem in several variables}
%\section{Introduction}

\section{Motivation}

%Given a commutative group $(G,+)$, a nonempty set $Y$, and a function $f:G\to Y$, we consider the set of periods of $f$, $\mathfrak{P}_0(f)=\{g\in G:f(w+g)=f(w)\text{ for all } w\in G\}$. Obviously, $\mathfrak{P}_0(f)$ is always a %subgroup of $G$ and, in some special cases, these groups are well known and, indeed, have a nice structure.  For example, 

A famous result proved by Jacobi in 1834 claims that  if $f:\mathbb{C}\to\widehat{\mathbb{C}}$ is a non constant meromorphic function defined on the complex numbers, then  $\mathfrak{P}_0(f)=\{w\in\mathbb{C}:f(z+w)=f(z) \text{ for all } z\in\mathbb{C}\}$, the set of periods of $f$,  is a discrete subgroup of $(\mathbb{C},+)$. This reduces the possibilities to the following three cases: $\mathfrak{P}_0(f)=\{0\}$, or $\mathfrak{P}_0(f)=\{nw_1:n\in\mathbb{Z}\}$ for a certain complex number $w_1\neq 0$, or $\mathfrak{P}_0(f)=\{n_1w_1+n_2w_2:(n_1,n_2)\in \mathbb{Z}^2\}$ for certain complex numbers $w_1,w_2$ satisfying $w_1w_2\neq 0$ and $w_1/w_2\not\in\mathbb{R}$. In particular, these functions cannot have three independent periods and there exist meromorphic functions  $f:\mathbb{C}\to\widehat{\mathbb{C}}$  with two independent periods $w_1,w_2$ as soon as $w_1/w_2\not\in\mathbb{R}$. These functions are called doubly periodic (or elliptic) and have an important role in complex function theory \cite{JS}.  Analogously, if the function $f:\mathbb{R}\to\mathbb{R}$ is continuous and non constant, it does not admit  two 
$\mathbb{Q}$-linearly independent periods. 

Obviously,  Jacobi's theorem can be formulated as a result which characterizes the constant functions as those meromorphic  functions $f:\mathbb{C}\to\widehat{\mathbb{C}}$  which solve a system of functional equations of the form 
\begin{equation}\label{JC}
\Delta_{h_1}f(z)=\Delta_{h_2}f(z)=\Delta_{h_3}f(z)=0 \ \ (z\in \mathbb{C})
\end{equation}
for three independent periods $\{h_1,h_2,h_3\}$ (i.e., $h_1\mathbb{Z}+h_2\mathbb{Z}+h_3\mathbb{Z}$ is a dense subset of  $\mathbb{C}$). For the real case, the result states that, if $h_1,h_2\in\mathbb{R}\setminus\{0\}$ are two nonzero real numbers and $h_1/h_2\not\in\mathbb{Q}$,  the continuous function $f:\mathbb{R}\to\mathbb{R}$ is a constant function if and only if it solves the system of functional equations 
\begin{equation} \label{JR}
\Delta_{h_1}f(x)=\Delta_{h_2}f(x)= 0\ (x\in \mathbb{R}).
\end{equation}
In 1937  Montel \cite{montel} proved an interesting nontrivial generalization of Jacobi's theorem. Concretely, he substituted in the equations $(\ref{JC}),(\ref{JR})$ above the first difference operator $\Delta_h$ by  the higher differences operator $\Delta^{m+1}_h$ (which is inductively defined by $\Delta_h^{n+1}f(x)=\Delta_h(\Delta_h^nf)(x)$, $n=1,2,\cdots$) and proved that these equations are appropriate for the characterization of ordinary polynomials. In particular, he proved the following result:
\begin{theorem}[Montel] \label{monteldimuno} Assume that $f:\mathbb{C}\to\mathbb{C}$ is an analytic function  which solves a system of functional equations of the form 
\begin{equation}\label{JCM}
\Delta_{h_1}^{m+1}f(z)=\Delta_{h_2}^{m+1}f(z)=\Delta_{h_3}^{m+1}f(z)=0 \ \ (z\in \mathbb{C})
\end{equation}
for three independent periods $\{h_1,h_2,h_3\}$. Then $f(z)=a_0+a_1z+\cdots+a_mz^m$ is an ordinary polynomial with complex coefficients and degree $\leq m$. Furthermore, if $\{h_1,h_2\}\subset \mathbb{R}\setminus \{0\}$ satisfy $h_1/h_2\not\in\mathbb{Q}$, 
 the continuous function $f:\mathbb{R}\to\mathbb{R}$ is an ordinary polynomial with real coefficients and degree $\leq m$ if and only if it solves the system of functional equations 
\begin{equation} \label{JRM}
\Delta_{h_1}^{m+1}f(x)=\Delta_{h_2}^{m+1}f(x)= 0\ (x\in \mathbb{R}).
\end{equation}
\end{theorem}

The functional equation $\Delta_h^{m+1}f(x)=0$ had already been introduced in the literature by M. Fr\'{e}chet in 1909 as a particular case of the functional equation  
%Concretely, in his seminal paper \cite{frechet}, he proved that the 
%Let $X, Y$ be two $\mathbb{Q}$-vector spaces and let $f:X\to Y$. We say that $f$ satisfies Fr\'{e}chet's functional equation of order $s-1$ if 
\begin{equation}\label{fre}
\Delta_{h_1h_2\cdots h_{m+1}}f(x)=0 \ \ (x,h_1,h_2,\dots,h_{m+1}\in \mathbb{R}),
\end{equation}
where $f:\mathbb{R}\to\mathbb{R}$ and $\Delta_{h_1h_2\cdots h_s}f(x)=\Delta_{h_1}\left(\Delta_{h_2\cdots h_s}f\right)(x)$, $s=2,3,\cdots$. In particular, after Fr\'{e}chet's 
seminal paper \cite{frechet}, the solutions of \eqref{fre} are named ``polynomials'' by the Functional Equations community, since it is known that, under very mild regularity conditions on $f$, if $f:\mathbb{R}\to\mathbb{R}$ satisfies \eqref{fre}, then $f(x)=a_0+a_1x+\cdots a_{m}x^{m}$ for all $x\in\mathbb{R}$ and certain constants $a_i\in\mathbb{R}$. For example, in order to have this property, it is enough for $f$ being locally bounded \cite{frechet}, \cite{almira_antonio}, but there are stronger results \cite{ger1}, \cite{kuczma1}, \cite{mckiernan}, \cite{popa_rasa}. The equation \eqref{fre} can be studied for functions $f:X\to Y$  whenever $X, Y$ are two  $\mathbb{Q}$-vector spaces and the variables $x,h_1,\cdots,h_{m+1}$ are assumed to be elements of $X$:
\begin{equation}\label{fregeneral}
\Delta_{h_1h_2\cdots h_{m+1}}f(x)=0 \ \ (x,h_1,h_2,\dots,h_{m+1}\in X).
\end{equation}
In this context, the general solutions of \eqref{fregeneral} are characterized as functions of the form $f(x)=A_0+A_1(x)+\cdots+A_m(x)$, where $A_0$ is a constant and $A_k(x)=A^k(x,x,\cdots,x)$ for a certain $k$-additive symmetric function $A^k:X^k\to Y$ (we say that $A_k$ is the diagonalization of $A^k$). In particular, if $x\in X$ and $r\in\mathbb{Q}$, then $f(rx)=A_0+rA_1(x)+\cdots+r^mA_m(x)$. Furthermore, it is known that $f:X\to Y$ satisfies \eqref{fregeneral} if and only if it satisfies 
\begin{equation}\label{frepasofijo}
\Delta_{h}^{m+1}f(x):=\sum_{k=0}^{m+1}\binom{m+1}{k}(-1)^{m+1-k}f(x+kh)=0 \ \ (x,h\in X).
\end{equation}
A proof of this fact follows directly from Djokovi\'{c}'s Theorem \cite{Dj} (see also \cite[Theorem 7.5, page 160]{HIR}, \cite[Theorem 15.1.2., page 418]{kuczma}), which states that  the operators $\Delta_{h_1 h_2\cdots h_s}$ satisfy the equation
\begin{equation}\label{igualdad}
\Delta_{h_1\cdots h_s}f(x)=
\sum_{\epsilon_1,\dots,\epsilon_s=0}^1(-1)^{\epsilon_1+\cdots+\epsilon_s}
\Delta_{\alpha_{(\epsilon_1,\dots,\epsilon_s)}(h_1,\cdots,h_s)}^sf(x+\beta_{(\epsilon_1,\dots,\epsilon_s)}(h_1,\cdots,h_s)),
\end{equation}
where $$\alpha_{(\epsilon_1,\dots,\epsilon_s)}(h_1,\cdots,h_s)=(-1)\sum_{r=1}^s\frac{\epsilon_rh_r}{r}$$ and $$\beta_{(\epsilon_1,\dots,\epsilon_s)}(h_1,\cdots,h_s)=\sum_{r=1}^s\epsilon_rh_r.$$  

In his seminal paper \cite{montel}, Montel also studied the equation $\eqref{frepasofijo}$ for $X=\mathbb{R}^d$, with $d>1$,  and $f:\mathbb{R}^d\to\mathbb{C}$ continuous, and for $X=\mathbb{C}^d$ and $f:\mathbb{C}^d\to \mathbb{C}$ analytic. Concretely, he stated (and gave a proof for $d=2$) the following result.
\begin{theorem}[Montel's Theorem in several variables] \label{montelvvcf}  Let $\{h_1,\cdots,h_{\ell}\}\subset \mathbb{R}^d$  be such that 
\begin{equation} \label{perm}
h_1\mathbb{Z}+h_2\mathbb{Z}+\cdots+h_{\ell}\mathbb{Z} \text{ is a dense subset of } \mathbb{R}^d,
\end{equation}
and let $f\in C(\mathbb{R}^d,\mathbb{C})$ be such that $\Delta_{h_k}^m(f) =0$, $k=1,\cdots,\ell$. Then  $f(x)=\sum_{|\alpha|<N}a_{\alpha}x^{\alpha}$ for some $N\in\mathbb{N}$, some complex numbers $a_{\alpha}$, and all $x\in\mathbb{R}^d$. Thus, $f$ is an ordinary complex valued polynomial in $d$ real variables. 

Consequently, if $d=2k$, $\{h_i\}_{i=1}^{\ell}$ satisfies $\eqref{perm}$, the function  $f:\mathbb{C}^k\to \mathbb{C}$ is holomorphic and $\Delta_{h_k}^m(f) =0$, $k=1,\cdots,\ell$, then  $f(z)=\sum_{|\alpha|<N}a_{\alpha}z^{\alpha}$ is an ordinary complex valued polynomial in $k$ complex variables. 
\end{theorem}

\begin{remark} The finitely generated  subgroups of $(\mathbb{R}^d,+)$ which are dense in $\mathbb{R}^d$ have been deeply studied, and many characterizations of them are known \cite[Proposition 4.3]{W}. For example, a theorem by Kronecker guarantees that, given $\theta_1,\theta_2,\cdots,\theta_d\in\mathbb{R}$, the group $\mathbb{Z}^d+(\theta_1,\theta_2,\cdots,\theta_d)\mathbb{Z}$ (which is generated by $d+1$ elements) is dense in 
$\mathbb{R}^d$ if and only if, 
\[
n_0+n_1\theta_1+\cdots+n_d\theta_d\neq 0, \text{ for every }(n_0,n_1,\cdots,n_d)\in\mathbb{Z}^{d+1}
\]  
(see, e.g., \cite[Theorem 4.1]{W}, for the proof of this result). 
\end{remark}

%Given $D\subseteq X$, the function $f:D\to Y$ is named a ``polynomial on $D$'' if $f$ satisfies \eqref{frepasofijo} for a certain $m\in\mathbb{N}$ and all $x,h\in X$ such that $\{x,x+h,\cdots,x+(m+1)h\}\subseteq D$. A natural %problem, that has been solved by R. Ger \cite{ger}, is to study the conditions under which a polynomial $f$ on $D$ can be extended (and how to make this) to a polynomial on $X$. 

%over a non Archimedian field $\mathbb{K}$. 

In \cite{almira_invariantes} the author  used the structure of finite dimensional translation invariant subspaces of $C(\mathbb{R},\mathbb{C})$ to give a new proof of Theorem \ref{monteldimuno}. In this paper we use similar ideas to prove a Montel's type theorem for the case of complex valued functions defined over the discrete group $\mathbb{Z}^d$. Furthermore, we  also state and demonstrate an improved version of  Theorem \ref{montelvvcf}. 

In all the paper we use the following standard notation: If $\alpha=(\alpha_1,\cdots,\alpha_d)\in\mathbb{N}^d$, $x=(x_1,\cdots,x_d)\in \mathbb{C}^d, \lambda=(\lambda_1,\cdots,\lambda_d)\in (\mathbb{C}\setminus\{0\})^d$, and $n=(n_1,\cdots,n_d)\in\mathbb{Z}^d$, then $n^{\alpha}=n_1^{\alpha_1}n_2^{\alpha_2}\cdots n_d^{\alpha_d}$, $x^{\alpha}=x_1^{\alpha_1}x_2^{\alpha_2}\cdots x_d^{\alpha_d}$, $\lambda^{n}=\lambda_1^{n_1}\cdots\lambda_d^{n_d}$, $|\alpha|=\sum_{k=1}^d \alpha_k$. $\Pi_{m}^{d}$ denotes the set of complex polynomials in $d$ variables with total degree $\leq m$ (when $d=1$ we write $\Pi_m$ instead of $\Pi_{m}^1$). Finally, 
 $f\in C(\mathbb{Z}^d,\mathbb{C})$ is named an ``exponential monomial'' if there exists a polynomial in the variables $n_1,\cdots,n_d\in\mathbb{Z}$, $p(n)=\sum_{|\alpha|\leq N}a_{\alpha}n^{\alpha}$, and a vector $\lambda\in (\mathbb{C}\setminus\{0\})^d$, such that $f(n)=p(n)\lambda^n$, for all $n\in\mathbb{Z}^d$. 

\section{Main resuts}

For the discrete case, we will need to use the following well known result by M. Lefranc, whose proof is based on algebraic geometry arguments (see, e.g., \cite{ron},  \cite{lefranc}, \cite{laszlo}).
  
\begin{theorem}[Lefranc, 1958] Assume that $V$ is a closed vector subspace of $C(\mathbb{Z}^d,\mathbb{C})$ which is invariant by translations, and let $\Gamma_V$ denote the set of exponential monomials which belong to $V$. Then $V=\overline{\mathbf{span}(\Gamma_V)}$. 
\end{theorem}

%Let us recall, for the sake of completeness, that $f\in C(\mathbb{Z}^d,\mathbb{C})$ is named an ``exponential monomial'' if there exists a polynomial in the variables $n_1,\cdots,n_d\in\mathbb{Z}$, $p(n)=\sum_{|\alpha|\leq N}% a_{\alpha}n^{\alpha}$, and a vector $\lambda\in (\mathbb{C}\setminus\{0\})^d$, such that $f(n)=p(n)\lambda^n$, for all $n\in\mathbb{Z}^d$. 

% which belong to $V$, where $n=(n_1,n_2,\cdots,n_d)\in \mathbb{Z}^d$, etc.)

An immediate consequence of Lefranc's Theorem is the following 

\begin{corollary} \label{cor_uno}
If $V$ is a finite dimensional vector subspace of $C(\mathbb{Z}^d,\mathbb{C})$ which is invariant by translations, then $V\subseteq W=\mathbf{span}(\bigcup_{k=0}^s\{n^\alpha\lambda_k^n: |\alpha|\leq m_k-1\})$ for  certain $s\in\mathbb{N}$,  $\{\lambda_k\}_{k=0}^s\subseteq (\mathbb{C}\setminus\{0\})^d$ and $\{m_k\}_{k=0}^s\subseteq \mathbb{N}$. 
\end{corollary}

\begin{remark} Here and, in all what follows, we assume that $\lambda_0=(1,\cdots,1)$ and that $m_0=0$ means that there are no elements of the form $n^{\alpha}$ in the basis.  
We always assume that $m_k\geq 1$ for $k=1,\cdots,s$. 
\end{remark}

Let us now state two technical results, which are extremely important for our arguments in this section.
%whose proof we include for the sake of completeness:

\begin{lemma} \label{uno} Let $E$ be a vector space and $L:E\to E$ be a linear operator defined on $E$. If $V\subset E$ is an $L^m$-invariant subspace of $E$, then the space
\[
\Box_L^m(V)=V+L(V)+L^2(V)+\cdots+L^m(V)
\] 
is $L$-invariant. Furthermore, $\Box_L^m(V)$ is the smallest $L$-invariant subspace of $E$ containing $V$.
\end{lemma}

\begin{proof}  The linearity of $L$ implies that 
\[
L(\Box_L^m(V))=L(V)+L^2(V)+L^3(V)+\cdots+L^m(V)+L^{m+1}(V).
\]
Now, $L^{m+1}(V)=L(L^{m}(V))\subseteq L(V)$ and $L(V)+L(V)=L(V)$, so that $L(\Box_L^m(V))\subseteq \Box_L^m(V)$. 

On the other hand, let us assume that $V\subseteq F\subseteq E$ and $F$ is an $L$-invariant subspace of $E$. If $\{v_k\}_{k=0}^m\subseteq V$, then 
$L^k(v_k)\in F$ for all $k\in \{0,1,\cdots,m\}$, so that $v_0+L(v_1)+\cdots+L^m(v_m)\in F$. This proves that $\Box_L^m(V)\subseteq F$.
\end{proof}

\begin{lemma} \label{dos} Let $E$ be a vector space and $L_1,L_2,\cdots,L_t:E\to E$ be linear operators defined on $E$. Assume that $L_iL_j=L_jL_i$ for all $i\neq j$.  If $V\subset E$ is a vector subspace of $E$ which satisfies  $\bigcup_{i=1}^tL_i^m(V)\subseteq V$, then 
\[
\diamond_{L_1,L_2,\cdots,L_t}^m(V)= \Box_{L_t}^m(\Box_{L_{t-1}}^m(\cdots (\Box_{L_1}^m(V))\cdots)) 
\]
is $L_i$-invariant for $i=1,2,\cdots, t$, and contains $V$. 
\end{lemma}

\begin{proof}
We proceed by induction on $t$. For $t=1$ the result follows from Lemma \ref{uno}. Assume the result holds true for $t-1$. Obviously, 
\[
\diamond_{L_1,L_2,\cdots,L_t}^m(V)=  \Box_{L_t}^m(\diamond_{L_1,L_2,\cdots,L_{t-1}}^m(V)). 
\]
By definition, 
\begin{eqnarray*}
&\ & L_t^m(\diamond_{L_1,L_2,\cdots,L_{t-1}}^m(V)) \\ 
& \ & \ \ =  
L_t^m(\Box_{L_{t-1}}^m(\diamond_{L_1,L_2,\cdots,L_{t-2}}^m(V)))\\
& \ & \ \ =   L_t^m(\diamond_{L_1,L_2,\cdots,L_{t-2}}^m(V)+L_{t-1}(\diamond_{L_1,L_2,\cdots,L_{t-2}}^m(V))+\cdots+
L_{t-1}^m(\diamond_{L_1,L_2,\cdots,L_{t-2}}^m(V)))\\
& \ & \ \ =   L_t^m(\diamond_{L_1,L_2,\cdots,L_{t-2}}^m(V))+L_t^m(L_{t-1}(\diamond_{L_1,L_2,\cdots,L_{t-2}}^m(V))) \\
& \ & \ \  \ \  \  \  \  \   +\cdots+
L_t^m(L_{t-1}^m(\diamond_{L_1,L_2,\cdots,L_{t-2}}^m(V)))\\
& \ & \ \ =   L_t^m(\diamond_{L_1,L_2,\cdots,L_{t-2}}^m(V))+L_{t-1}(L_t^m(\diamond_{L_1,L_2,\cdots,L_{t-2}}^m(V))) \\
& \ & \  \  \ \  \  \  \  \    +\cdots+
L_{t-1}^m(L_t^m(\diamond_{L_1,L_2,\cdots,L_{t-2}}^m(V)))\\
& \ & \ \ =   \Box_{L_{t-1}}^m(L_t^m(\diamond_{L_1,L_2,\cdots,L_{t-2}}^m(V))),
\end{eqnarray*}
since  $L_t,L_{t-1}$  commute. Repeating this process, we get
\begin{eqnarray*}
&\ & L_t^m(\diamond_{L_1,L_2,\cdots,L_{t-1}}^m(V)) \\ 
& \ & \ \ =   \Box_{L_{t-1}}^m(L_t^m(\diamond_{L_1,L_2,\cdots,L_{t-2}}^m(V))) \\
& \ & \ \ =   \Box_{L_{t-1}}^m(  \Box_{L_{t-2}}^m(L_t^m(\diamond_{L_1,L_2,\cdots,L_{t-3}}^m(V)))) \\
& \ & \ \  \vdots \\
& \ & \ \ =   \Box_{L_{t-1}}^m(  \Box_{L_{t-2}}^m( \cdots  (\Box_{L_{1}}^m( L_t^m(V)))\cdots )) \\
& \ & \ \  \subseteq   \Box_{L_{t-1}}^m(  \Box_{L_{t-2}}^m( \cdots  (\Box_{L_{1}}^m( V))\cdots )) \\
& \ & \ \ = \diamond_{L_1,L_2,\cdots,L_{t-1}}^m(V), 
\end{eqnarray*}
since $L_tL_i=L_iL_t$ for all $i<t$ and $L_t^m(V)\subseteq V$. This proves that $\diamond_{L_1,L_2,\cdots,L_{t-1}}^m(V)$ is $L_t^m$-invariant, and Lemma \ref{uno} implies that $\diamond_{L_1,L_2,\cdots,L_t}^m(V)=  \Box_{L_t}^m(\diamond_{L_1,L_2,\cdots,L_{t-1}}^m(V))$ is $L_t$-invariant. 

On the other hand, given $i<t$, the identity $L_tL_i=L_iL_t$, and the fact that $$L_i(\diamond_{L_1,L_2,\cdots,L_{t-1}}^m(V))\subseteq \diamond_{L_1,L_2,\cdots,L_{t-1}}^m(V)$$ (which follows from the hypothesis of induction) imply that
\begin{eqnarray*}
& \ & \ \ L_i(\diamond_{L_1,L_2,\cdots,L_t}^m(V)) \\
& \ & \ \ =    L_i(\Box_{L_t}^m(\diamond_{L_1,L_2,\cdots,L_{t-1}}^m(V)))\\
& \ & \ \ =   L_i(\diamond_{L_1,L_2,\cdots,L_{t-1}}^m(V)+ L_t(\diamond_{L_1,L_2,\cdots,L_{t-1}}^m(V) )+\cdots +L_t^m( \diamond_{L_1,L_2,\cdots,L_{t-1}}^m(V)) ) \\
& \ & \ \ =   L_i(\diamond_{L_1,L_2,\cdots,L_{t-1}}^m(V))+ L_t(L_i(\diamond_{L_1,L_2,\cdots,L_{t-1}}^m(V) ))+\cdots +L_t^m(L_i( \diamond_{L_1,L_2,\cdots,L_{t-1}}^m(V)) ) \\
& \ & \ \  \subseteq    \diamond_{L_1,L_2,\cdots,L_{t-1}}^m(V)+ L_t(\diamond_{L_1,L_2,\cdots,L_{t-1}}^m(V))+\cdots +L_t^m(\diamond_{L_1,L_2,\cdots,L_{t-1}}^m(V) ) \\
& \ & \ \ =   \Box_{L_t}^m(\diamond_{L_1,L_2,\cdots,L_{t-1}}^m(V))\\
& \ & \ \ =  \diamond_{L_1,L_2,\cdots,L_t}^m(V),
\end{eqnarray*} 
so  that $\diamond_{L_1,L_2,\cdots,L_t}^m(V)$ is $L_i$-invariant for all $i\leq t$. Finally, it is clear that $V\subseteq \Box_{L_1}^m(V) \subseteq \cdots \subseteq  \diamond_{L_1,L_2,\cdots,L_t}^m(V)$.
\end{proof}

It is important to note that there are many examples of linear transformations $T:E\to E$ such that $T$ and $T^m$ have different sets of invariant subspaces. For example, if $T$ is not of the form $T=\lambda I$ for any scalar $\lambda$ and satisfies $T^m=I$ or $T^m=0$, then all subspaces of $E$ are invariant subspaces of $T^m$  and, on the other hand, there exists $v\in E$ such that $Tv\not\in \mathbf{span}\{v\}$, so that $\mathbf{span}\{v\}$ is not an invariant subspace of $T$.  On the other hand, as the following lemma proves, sometimes it is possible to show that $T$ and $T^m$ share the same set of invariant subspaces.

\begin{lemma} \label{nuevo} Let  $\mathbb{K}$ be a field and  $E$ be a  $\mathbb{K}$-vector space with basis  $\beta=\{v_k\}_{k=1}^n$ and let $m\in\mathbb{N}$, $m\geq 1$. Assume that $T:E\to E$ is such that $A=M_{\beta}(T)$ is of the form $A=\lambda I+B$, where $\lambda\in \mathbb{C}$ and $B$ is strictly upper triangular with nonzero entries in the first superdiagonal. Then the full list of $T$-invariant subspaces of $E$ is given by $V_0=\{0\}$ and $V_k=\mathbf{span}\{v_1,\cdots,v_k\}$, $k=1,2,\cdots,n$. Furthermore, if $\lambda \neq 0$, then $T^m$ has the same invariant subspaces as $T$. \end{lemma}

\begin{proof} Assume that $A=M_{\beta}(T)$ is of the form $A=\lambda I+B$, where $\lambda\neq 0$ and $B$ is strictly upper triangular with nonzero entries in the first superdiagonal, and let $V\neq \{0\}$ be a $T$-invariant subspace.  Let $v\in V$, $v=a_1v_1+\cdots+a_sv_s$, $a_s\neq 0$. Then $w=Tv-\lambda v\in V$ and a simple computation shows that $w=\alpha_1v_1+\cdots+\alpha_{s-1}v_{s-1}$ with $\alpha_{s-1}=b_{s-1,s}a_s\neq 0$, where $B=(b_{ij})_{i,j=1}^n$. It follows that, if $V$ is $T$-invariant and  $v=a_1v_1+\cdots+a_sv_s\in V$ with $a_s\neq 0$, then $\mathbf{span}\{v_1,v_2,\cdots,v_s\}\subseteq V$. Take $k_0=\max\{k:\text{ exists }v\in V, v=a_1v_1+\cdots+a_sv_s\text{ and }a_s\neq 0\}$. Then $V=\mathbf{span}\{v_1,\cdots,v_{k_0}\}$. Finally, it is clear that all the spaces $V_k=\mathbf{span}\{v_1,\cdots,v_k\}$, $k=1,2,\cdots,n$ are $T$-invariant.

Let us now assume that $\lambda\neq 0$. To compute the invariant subspaces of $T^m$ we take into account that $A^m=M_{\beta}(T^m)$ and 
\begin{eqnarray*}
A^m &=& (\lambda I+B)^m\\
&=& \sum_{k=0}^m\binom{m}{k}(-1)^{m-k}B^k\\
&=& \lambda^mI+m\lambda^{m-1}B+\sum_{k=2}^m\binom{m}{k}(-1)^{m-k}B^k.
\end{eqnarray*}
This shows that $A^m=\lambda^mI+C$, with $C$ strictly upper triangular  with nonzero entries in the first superdiagonal, since the only contribution  to the first superdiagonal  of  $C=m\lambda^{m-1}B+\sum_{k=2}^m\binom{m}{k}(-1)^{m-k}B^k$ is got from $m\lambda^{m-1}B$, and $\lambda\neq 0$ . It follows that we can apply the first part of the lemma to the linear transformation $T^m$, which concludes the proof. 
\end{proof}

\begin{proposition}\label{VdentroW} Assume that $V$ is a finite dimensional subspace of $C(\mathbb{Z}^d,\mathbb{C})$, $\{h_1,\cdots,h_{t}\}\subset \mathbb{Z}^d$  and $h_1\mathbb{Z}+h_2\mathbb{Z}+\cdots+h_t\mathbb{Z}=\mathbb{Z}^d$. If $\Delta_{h_k}^m(V)\subseteq V$, $k=1,\cdots,t$,  then there exists a
finite dimensional subspace $W$ of  $C(\mathbb{Z}^d,\mathbb{C})$ which is invariant by translations and contains $V$. Consequently, all elements of $V$ are exponential polynomials, 
$f(n)=\sum_{k=0}^s(\sum_{|\alpha| \leq m_k}a_{k,\alpha}n^\alpha) \lambda_k^n$
\end{proposition}

\begin{proof} We apply Lemma \ref{dos} with $E=C(\mathbb{Z}^d,\mathbb{C})$, $L_i=\Delta_{h_i}$, $i=1,\cdots,t$, to conclude that 
$V\subseteq W=\diamond_{\Delta_{h_1},\Delta_{h_2},\cdots,\Delta_{h_t}}^m(V)$ and $W$ is a finite dimensional subspace of $\mathbf{C}(\mathbb{Z}^d)$ satisfying $\Delta_{h_i}(W)\subseteq W$ , $i=1,2,\cdots,t$. Hence $W$ is invariant by translations, since $h_1\mathbb{Z}+h_2\mathbb{Z}+\cdots+h_t\mathbb{Z}=\mathbb{Z}^d$. Applying  Corollary \ref{cor_uno}, we conclude that  all elements of $W$ (hence, also all elements of $V$) are exponential polynomials.
\end{proof}
 
\begin{theorem}[Discrete Montel's Theorem] \label{discretemonteltheorem} Let $\{h_1,\cdots,h_{t}\}\subset \mathbb{Z}^d$  be such that $h_1\mathbb{Z}+h_2\mathbb{Z}+\cdots+h_t\mathbb{Z}=\mathbb{Z}^d$, and let $f\in C(\mathbb{Z}^d,\mathbb{C})$. If $\Delta_{h_k}^m(f) =0$, $k=1,\cdots,t$, then  $f(n)=\sum_{|\alpha|<N}a_{\alpha}n^{\alpha}$ for some $N\in\mathbb{N}$, some complex numbers $a_{\alpha}$, and all $n\in\mathbb{Z}^d$. In other words:  $f$ is an ordinary polynomial on $\mathbb{Z}^d$. Furthermore, if $d=1$, then $f(n)=a_0+a_1n+\cdots+a_{m-1}n^{m-1}$ is an ordinary polynomial on $\mathbb{Z}$, of degree $\leq m-1$.
\end{theorem}

\begin{proof} Assume that  $\Delta_{h_k}^m(f) =0$, $k=1,\cdots,t$. Then $V=\mathbf{span}\{f\}$ is a one dimensional subspace of $C(\mathbb{Z}^d,\mathbb{C})$ which   
satisfies the hypotheses of Proposition \ref{VdentroW}. Hence all elements of $V$ are exponential polynomials. In particular, $f$ is an exponential polynomial,
\begin{equation} \label{efe}
f(n)=\sum_{k=0}^s(\sum_{|\alpha| \leq m_k}a_{k,\alpha}n^\alpha) \lambda_k^n
\end{equation}
and we can assume, with no loss of generality,  that $\lambda_0=(1,1,\cdots,1)$, $m_0\geq m-1$, and $\lambda_i\neq\lambda_j$ for all $i\neq j$. 

Let 
\begin{equation} \label{base}
\beta = \{
 n^{\alpha}\lambda_k^n, \ \  0\leq |\alpha|\leq m_k \text{ and } k=0,1,2,\cdots,s\} 
\end{equation}
and $\mathcal{E}=\mathbf{span}\{\beta\}$ be an space with a basis of the form \eqref{base} which contains $V$.  Let us consider the linear map $\Delta_{h}:\mathcal{S}\to\mathcal{S}$ induced by the operator $\Delta_h$ when restricted to $\mathcal{E}$.  Obviously, $\mathcal{E}=P \oplus E_1\oplus E_2\oplus \cdots\oplus E_s$, where 
\[
 P=\mathbf{span}\{n^{\alpha}\}_{0\leq |\alpha|\leq m_0} \text{ and } E_k=\mathbf{span}\{n^{\alpha}\lambda_k^n\} _{0\leq |\alpha|\leq m_k},  \  k=1,2,\cdots,s. 
\]
Furthermore, $\Delta_h(P)\subseteq P$ and $\Delta_h(E_k)\subseteq E_k$ for $k=1,2,\cdots,s$, since
\[
q_{\alpha}(n)=\Delta_hn^{\alpha}=(n+h)^\alpha-n^\alpha
\]
is a polynomial of degree $\leq |\alpha|-1$ and 
\[
\Delta_h(n^{\alpha}\lambda_k^n)=((n+h)^\alpha \lambda_k^h-n^\alpha)\lambda_k^n=(n^\alpha( \lambda_k^h-1)+q_{\alpha}(n))\lambda_k^n.
\]
It follows that, for any $h\in\mathbb{Z}^d$, the operator $\Delta_h^m$ also satisfies $\Delta_h(P)\subseteq P$ and $\Delta_h(E_k)\subseteq E_k$ for $k=1,2,\cdots,s$, so that, if $g\in\mathcal{E}$, then 
$\Delta_h^mg=0$ if and only if $\Delta_h^mp=0$, $\Delta_h^mb_k=0$, $k=1,\cdots,s$, where $g=p+b_1+\cdots+b_s$, $p\in P$, $b_k\in E_k$, $k=1,\cdots,s$. 

Let us fix $k\in \{1,\cdots,s\}$ and let us consider the operator $(\Delta_h)_{|E_k}:E_k\to E_k$. The matrix $A_k$ associated to this operator with respect to the basis $\beta_k= \{n^{\alpha}\lambda_k^n\} _{0\leq |\alpha|\leq m_k}$, which we consider ordered by the graduated lexicographic order, 
\[
n^{\alpha}\lambda_k^n\leq_{grlex} n^{\gamma}\lambda_k^n \Leftrightarrow \left( |\alpha|\leq |\gamma| \text{ or }(|\alpha|=|\gamma| \text{ and } \alpha\leq_{lex} \gamma) \right),
\] 
is upper triangular and the terms in the diagonal are all equal to  $d_k(h)=\lambda_k^h-1$ (Recall that $ \alpha\leq_{lex} \gamma$ if and only if $\alpha_{k_0}-\gamma_{k_0}<0$, where $k_0=\max\{k\in\{1,2,\cdots,d\}:\alpha_k\neq \gamma_k\}$). Obviously, these computations imply that, if $d_k(h)\neq 0$, then $(\Delta_h)_{|E_k}$ is invertible and, in particular, $\Delta_hb_k=0$ and $b_k\in E_k$ imply $b_k=0$. 

Let us now study the equations $d_k(h_j)=0$, with  $h_j=(h_{j,1},\cdots,h_{j,d})$, $j=1,\cdots,t$,  being the vectors fixed by the hypotheses of the theorem. We know that $$\lambda_k=(\rho_{k,1}e^{2\pi \mathbf{i} \theta_{k,1}},\rho_{2,k}e^{2\pi \mathbf{i} \theta_{2,k}},\cdots,\rho_{k,d}e^{2\pi \mathbf{i} \theta_{k,d}})\in(\mathbb{C}\setminus\{0\})^d$$ with $\rho_{k,d}\neq 1$ for at least one of the values $k$, since $\lambda_k\neq (1,1,\cdots,1)$. Hence $d_k(h_j)=0$ if and only if 
\[
(\rho_{k,1}e^{2\pi \mathbf{i} \theta_{k,1}})^{h_{j,1}}(\rho_{k,2}e^{2\pi \mathbf{i} \theta_{k,2}})^{h_{j,2}}\cdots(\rho_{k,d}e^{2\pi \mathbf{i} \theta_{k,d}})^{h_{j,d}}=1,
\]
which is equivalent to the system of relations
\[
\left \{
\begin{array}{cccccc}
h_{j,1}\log\rho_{k,1}+h_{j,2}\log\rho_{k,2}+\cdots+h_{j,d}\log\rho_{k,d} = 0\\
 \theta_{k,1}h_{j,1}+\theta_{k,2}h_{j,2}+\cdots+\theta_{k,d}h_{j,d}\in \mathbb{Z}
 \end{array}\right. .
\]
Now, $h_1\mathbb{Z}+h_2\mathbb{Z}+\cdots+h_t\mathbb{Z}=\mathbb{Z}^d$, so that 
$\{h_1,h_2,\cdots,h_t\}$ contains a basis of $\mathbb{R}^d$. Furthermore,  $w_k=(\log\rho_{k,1},\log\rho_{k,2},\cdots,\log\rho_{k,d} )\in \mathbb{R}^d\setminus \{(0,0,\cdots,0)\}$. This implies that there exists $j_k\in \{1,\cdots,t\}$ such that $h_{j_k}$ is not orthogonal to  $w_k$. In particular, $d_k(h_{j_k})\neq 0$ and $(\Delta_{h_{j_k}})_{|E_k}$ is invertible. 

Consider the function $f$ given by $\eqref{efe}$. Then $f=p_0+b_1+\cdots+b_s\in\mathcal{E}$ (with $p_0 = \sum_{|\alpha| \leq m_0}a_{0,\alpha}n^\alpha\in P$ and $b_k=(\sum_{|\alpha| \leq m_k}a_{k,\alpha}n^\alpha) \lambda_k^n\in E_k$, $k=1,\cdots, s$) and $\Delta_{h_j}^mf=0$ for all $j$. For every $k\in 1,\cdots,s$ we have that $\Delta_{h_{j_k}}^mb_k=0$, which implies $b_k=0$, since    $(\Delta_{h_{j_k}})_{|E_k}$ is invertible. Hence $f=p_0$, which  proves the first part of the theorem. 

Let us now assume that $d=1$. We know that $f(n)=a_0+a_1n+\cdots+a_{m_0}n^{m_0}$ is an ordinary polynomial (with $m_0\geq m-1$) and we want to demonstrate that $\deg(f)\leq m-1$. To prove this assertion we fix our attention on the matrix $A$ associated to $\Delta_h:\Pi_{m_0}\to\Pi_{m_0}$ with respect to the basis $\beta_0=\{n^k\}_{k=0}^{m_0}$.  A simple computation shows that 
\begin{equation}\label{A0}
A=\left [
\begin{array}{cccccc}
0 & h & h^2 & \cdots & h^{m_0} \\
0 & 0 & 2h & \cdots & \binom{m_0}{1}h^{m_0-1}\\
\vdots & \vdots & \ddots & \cdots & \vdots \\
0 & 0 & 0 & \cdots &  \binom{m_0}{m_0-1}h\\
0 & 0 & 0 & \cdots &  0
\end{array} \right],
\end{equation}
so that the matrix associated to $(\Delta_h^m)_{|\Pi_{m_0}}$ with respect to the basis $\beta_0$ is given by $A^m$.
Now, $$\mathbf{ker}(A^m)=\mathbf{span}\{(0,0,\cdots,0,1^{\text{(i-th position)}},0,\cdots, 0): i=1,2,\cdots,m\}.$$ Hence 
$\mathbf{rank}(A^m)=m_0+1-m=\dim_{\mathbb{C}}(\Pi_{m_0})-m$ and $\dim_{\mathbb{C}} \mathbf{ker}(A^m)=m$. On the other hand, another simple computation shows that the space of ordinary polynomials of degree $\leq m-1$,  $\Pi_{m-1}$, is contained into 
$\mathbf{ker}(\Delta_h^m)$.  Hence $\mathbf{ker}(\Delta_h^m)=\Pi_{m-1}$, since both spaces have the same dimension. This, in conjunction with  $f\in \mathbf{ker}(\Delta_h^m)$, ends the proof. 

\end{proof}

\begin{corollary}
Let us assume that $h_1,h_2\in\mathbb{Z}$ are coprime numbers. If   $f\in C(\mathbb{Z},\mathbb{C})$ satisfies $\Delta_{h_k}^m(f) =0$, $k=1,2$, then $f(n)=a_0+a_1n+\cdots+a_{m-1}n^{m-1}$ is an ordinary polynomial of degree $\leq m-1$.
\end{corollary}
\begin{proof} If $h_1,h_2$ are coprime then, by B\'{e}zout's identity, there exists $a,b\in\mathbb{Z}$ such that $1=ah_1+bh_2$. Hence $h_1\mathbb{Z}+h_2\mathbb{Z}=\mathbb{Z}$ and the result follows from Theorem \ref{discretemonteltheorem}.
\end{proof}

The estimation of the degree of $f$ in Theorem \ref{discretemonteltheorem} when $d>1$ can be achieved in certain special cases. To prove a result of this type, we introduce the following technical result.

\begin{lemma} \label{gradopolinomios}
Let $p(x)=\sum_{|\alpha|\leq N}a_{\alpha}x^{\alpha}\in \mathbb{C}[x_1,x_2,\cdots,x_d]$ be a complex polynomial of $d$ complex variables with total degree $\leq N$. Assume that, for all $a=(a_1,\cdots, a_{i-1},a_{i+1},\cdots,a_d)\in\mathbb{C}^{d-1}$ and all $i\in\{1,\cdots,d\}$, the polynomial $g_{a,i}(t)=p(a_1,\cdots, a_{i-1},t,a_{i+1},\cdots,a_d)$ satisfies $g_{a,i}\in \Pi_m$. Then $N\leq md$. Furthermore, the extremal case is attained for $p(x)=x_1^mx_2^m\cdots x_d^m$.
\end{lemma}

\begin{proof} Let $i\in \{1,\cdots,d\}$ be fixed and let us assume that $\alpha=(\alpha_1,\cdots,\alpha_d)$ satisfies $a_{\alpha}\neq 0$ and $\alpha_i\geq m+1$. Then 
\[
p(x)=\sum_{k=0}^Nb_k(x_1,x_2,\cdots,x_{i-1},x_{i+1},\cdots,x_d)x_i^k; \text{ and } b_{\alpha_i}(x_1,\cdots,x_{i-1},x_{i+1},\cdots,x_d)\neq 0.
\]
In particular, there exists $a=(a_1,\cdots, a_{i-1},a_{i+1},\cdots,a_d)\in\mathbb{C}^{d-1}$ such that $b_{\alpha_i}(a)\neq 0$ since $b_{\alpha_i}$ is a nonzero polynomial. It follows that 
$$g_{a,i}(t)=p(a_1,\cdots, a_{i-1},t,a_{i+1},\cdots,a_d)=\sum_{k=0}^Nb_k(a_1,a_2,\cdots,a_{i-1},a_{i+1},\cdots,a_d)t^k$$ 
is a polynomial of degree bigger than $m$, which contradicts our hypotheses. Hence $a_{\alpha}\neq 0$ implies $\max\{\alpha_1,\alpha_2,\cdots,\alpha_d\}\leq m$ and $N\leq md$.   
\end{proof}

\begin{corollary} \label{control_grado_discreto}
Let $e_i=(0,\cdots,0,1^{i-\text{th position}},0,\cdots,0)\in \mathbb{Z}^d$, $i=1,\cdots,d$ and let us assume that $f\in C(\mathbb{Z}^d,\mathbb{C})$ satisfies $\Delta_{e_i}^mf=0$ for $i=1,\cdots,d$. Then $f(n)=\sum_{|\alpha|\leq (m-1)d}a_{\alpha}n^{\alpha}$  for all $n\in\mathbb{Z}^d$ and certain complex values $a_{\alpha}$. Furthermore, the extremal case is attained for $f(n)=n_1^{m-1}n_2^{m-1}\cdots n_d^{m-1}$.
\end{corollary}

\begin{proof} It follows from $e_1\mathbb{Z}+\cdots+e_d\mathbb{Z}=\mathbb{Z}^d$ and Theorem \ref{discretemonteltheorem} that  $f(n)=\sum_{|\alpha|<N}a_{\alpha}n^{\alpha}$ for some $N\in\mathbb{N}$, some complex numbers $a_{\alpha}$, and all $n\in\mathbb{Z}^d$. Let us introduce the polynomial $p(x)=\sum_{|\alpha|<N}a_{\alpha}x^{\alpha} \in \mathbb{C}[x_1,x_2,\cdots,x_d]$, and let us take $i\in\{1,\cdots,d\}$ and $a=(a_1,\cdots, a_{i-1},a_{i+1},\cdots,a_d)\in\mathbb{C}^{d-1}$.  Then $q_{a,i}(t)=p(a_1,\cdots, a_{i-1},t,a_{i+1},\cdots,a_d)$ is a polynomial in the complex variable $t$ and $\phi_{a,i}(x)=\Delta_{e_i}^{m}q_{a,i}(t)\in\mathbb{C}[t]$ satisfies $(\phi_{a,i})_{|\mathbb{Z}}=0$, so that $\phi_{a,i}=0$ and  $q_{a,i}$ is a polynomial of degree $\leq m-1$. The result follows by applying Lemma \ref{gradopolinomios} to $p$.  To prove the last claim of the corollary it is enough to check that $f(n)=n_1^{m-1}n_2^{m-1}\cdots n_d^{m-1}$ satisfies $\Delta_{e_i}^mf=0$ for $i=1,\cdots,d$, which is an easy exercise. 
\end{proof}

In \cite[Lemma 15.9.4.]{kuczma} it is proved that, if $f:\mathbb{R}^d\to \mathbb{R}$ is an ordinary polynomial separately in each one of its variables, and the partial degrees of $f$ are uniformly bounded by a certain natural number $m$, independently of the concrete variable and independently of the values of the other variables, then   $f$ is itself an ordinary polynomial. Corollary \ref{control_grado_discreto} proves that  an analogous result holds for functions $f:\mathbb{Z}^d\to \mathbb{C}$ and gives a concrete upper bound for the total degree of $f$. In \cite[Theorem 14]{prager} the authors proved that, if $\mathbb{K}$ is a field and $f:\mathbb{K}^d\to \mathbb{K}$  is an ordinary polynomial separately in each variable (but without any other assumption about the degrees of the polynomials appearing in this way), then $f$ is an ordinary polynomial jointly in all its variables, provided that $\mathbb{K}$ is finite or uncountable. Furthermore, for the case of $\mathbb{K}$ with infinite countable cardinal, they constructed a function $\chi:\mathbb{K}^d\to \mathbb{K}$ which is not a polynomial function on $\mathbb{K}^d$, but it is an ordinary polynomial separately in each variable. This construction can be translated to the case of functions $\mathbb{Z}^d\to \mathbb{K}$. Indeed, if $\mathbb{K}$ has characteristic zero, the result follows just by considering the function $\chi_{|\mathbb{Z}^d}$. An interesting open question is if these functions can also be constructed from $\mathbb{Z}^d$ into $\mathbb{Z}$.

Let $X_d$ denote either $C(\mathbb{R}^d, \mathbb{C})$ or the space of complex valued distributions defined on $\mathbb{R}^d$. The same kind of arguments we have used for the proof of Theorem \ref{discretemonteltheorem}, with small variations, jointly with Anselone-Korevaar's characterization of translation invariant finite dimensional subspaces of 
$X_d$  \cite{anselone} as the spaces admitting a basis of the form 
\begin{equation} \label{base2}
\beta =\{
 x^{(\alpha_{k,1},\cdots,\alpha_{k,d})}e^{< x,\lambda_k>}, \ \  0\leq \alpha_{k,i}\leq m_{k,i}-1 \text{ for }  i=1,\cdots,d  \text{ and } k=0,1,2,\cdots,s\}, 
\end{equation}
(where $\lambda_0=1$ and $m_{k,i}=0$ means that the variable $x_{i}$ does not appear in the polynomial part of the exponential monomials of the form $x^{\alpha}e^{<x,\lambda_k>}$) lead to a proof of the following result, which is an improvement of classical  Montel's theorem in several variables, since it is formulated for distributions:

\begin{theorem}[Montel's Theorem for distributions] \label{montelvv}  Let $\{h_1,\cdots,h_{\ell}\}\subset \mathbb{R}^d$  be such that  $ h_1\mathbb{Z}+ \cdots + h_{\ell}\mathbb{Z}$ is a dense subset of $\mathbb{R}^d$, 
%\begin{equation} \label{per} 
 %\mathbf{span}_{\mathbb{R}}\{h_1,\cdots,h_{d+1}\}=\mathbb{R}^d \text{ and  }
%h_i\not\in h_1\mathbb{Z}+ \cdots + h_{i-1}\mathbb{Z} +h_{i+1}\mathbb{Z}+\cdots+ h_{d+1}\mathbb{Z} \text{ for }
%i=1,\cdots, d+1,
%\end{equation} 
and let $f\in X_d$. If $\Delta_{h_k}^m(f) =0$, $k=1,\cdots,\ell$, then  $f(x)=\sum_{|\alpha|<N}a_{\alpha}x^{\alpha}$ for some $N\in\mathbb{N}$, some complex numbers $a_{\alpha}$, and all $x\in\mathbb{R}^d$. Thus, $f$ is an ordinary complex valued polynomial in $d$ real variables. Furthermore, if $d=1$, then $f(x)=a_1+a_1x+\cdots+a_{m-1}x^{m-1}$ is an ordinary polynomial of degree $\leq m-1$.
\end{theorem}

\begin{proof} A repetition of the proof of Proposition \ref{VdentroW}, but considering the operators $\Delta_{h_i}$ defined for distributions $f\in X_d$, shows that every finite dimensional subspace $V$ of $X_d$ which satisfies 
\begin{equation}
\Delta_{h_i}^m(V)\subseteq V \text{ for } i=1,\cdots,\ell,
\end{equation}
is included into an space $W$ which admits a basis of the form $\eqref{base2}$, since we can obtain a subspace $W\subseteq X_d$ which contains $V$ and it is invariant under the translations $\tau_{h_i}:X_d\to X_d$, $i=1,\cdots,\ell$, and, which implies that $W$ is invariant under all translations $\tau_{h}$ with $h\in  h_1\mathbb{Z}+ \cdots + h_{\ell}\mathbb{Z}$, which is a dense subset of $\mathbb{R}^d$. Hence $W$ is translation invariant and, thanks to Anselone-Korevaar's theorem, admits a basis of the form 
$\eqref{base2}$. Now we include the space $W$ into another space which admits a basis of the form
\begin{equation}\label{base3}
\beta = \{x^\alpha\}_{0\leq |\alpha|\leq m_0}\cup \{
 x^{\alpha}e^{< x,\lambda_k>}, \ \  0\leq |\alpha|\leq m_k \text{ and } k=1,2,\cdots,s\},
\end{equation}
since these bases are better suited for our computations. 
The proof ends taking $V=\mathbf{span}\{f\}$ with $\Delta_{h_k}^m(f) =0$, $k=1,\cdots,\ell$ and repeating the arguments of Theorem \ref{discretemonteltheorem} with small modifications.
\end{proof}

\begin{remark} In \cite[Page 78, Theorem 10.2]{laszlo1}, Anselone-Korevaar's theorem was generalized to the context of measurable complex valued functions defined on locally compact abelian groups. This produces the expectative of proving a new version of Theorem \ref{montelvv}, by changing the condition $f\in X_d$ by the hypothesis that  $f:\mathbb{R}^d\to\mathbb{C}$ is a measurable function.   In other words,  the question if Montel's theorem holds true for measurable functions arises in a natural way. Unfortunately, the answer is negative: Montel's theorem fails for measurable functions, for all $d\geq 1$. Indeed, let us assume that $f$ is measurable, $\Delta_{h_k}^m(f) =0$, $k=1,\cdots,\ell$, with  $\{h_1,\cdots,h_{\ell}\}\subset \mathbb{R}^d$  such that  $ h_1\mathbb{Z}+ \cdots + h_{\ell}\mathbb{Z}$ is a dense subset of $\mathbb{R}^d$, and let us apply the arguments in Lemma \ref{dos} to $V=\mathbf{span}\{f\}$ with $L_i=\Delta_{h_i}$. Then $V$ is contained into a finite dimensional space $W$ which is invariant under all translations of the form  
$\tau_h$ with $h\in   h_1\mathbb{Z}+ \cdots + h_{d+1}\mathbb{Z}$. But this does not imply that $W$ is translation invariant! Obviously, this explains why our argument fails. To prove that Montel's theorem fails for measurable functions we give a counterexample. Assume $d=1$ and $h_1,h_2\in\mathbb{R}\setminus\{0\}$, $h_1/h_2\not\in\mathbb{Q}$. Define $f(ph_1+qh_2)=pq$ for all $p,q\in\mathbb{Z}$, and $f(x)=0$ for all $x\not\in h_1\mathbb{Z}+h_2\mathbb{Z}$. Then $f$ is measurable and a simple computation shows that $\Delta_{h_1}^nf=\Delta_{h_2}^nf=0$ for all $n\geq 2$. On the other hand, $f$ is not a polynomial. In fact $f$ is also not a solution of Fr\'{e}chet's functional equation. Similar examples can be constructed for all $d>1$.
\end{remark}

 \bibliographystyle{amsplain}

%%%  ==============================================================

%Phone: (34)+ 953648503

%Fax: (34)+ 953648575}

%Phone: (34)+ 953648503

%Fax: (34)+ 953648575}

\bigskip

\footnotesize{J. M. Almira and Kh. F. Abu-Helaiel

Departamento de Matem\'{a}ticas. Universidad de Ja\'{e}n.

E.P.S. Linares,  C/Alfonso X el Sabio, 28

23700 Linares (Ja\'{e}n) Spain

Email: jmalmira@ujaen.es }

%Phone: (34)+ 953648503

%Fax: (34)+ 953648575}

%Phone: (34)+ 953648503

%Fax: (34)+ 953648575}

\end{document}